\documentclass{article}

\usepackage{arxiv}

\usepackage[utf8]{inputenc} 
\usepackage[T1]{fontenc}    
\usepackage{hyperref}       
\usepackage{url}            
\usepackage{booktabs}       
\usepackage{amsfonts}       
\usepackage{nicefrac}       
\usepackage{microtype}      
\usepackage{lipsum}
\usepackage{amsmath}
\usepackage[normalem]{ulem}

\usepackage{amsfonts,amssymb}
\usepackage{graphicx}
\usepackage{amsthm}
\usepackage{xcolor}

\newcommand{\R}{\mathbb{R}}
\newcommand{\s}{\mathbb{S}}

\newtheorem{theorem}{Theorem}[section]
\newtheorem{proposition}{Proposition}[section]
\newtheorem{lemma}{Lemma}[section]
\newtheorem{corollary}{Corollary}[section]

\usepackage{cite}
\usepackage{authblk}

\def\supp{\text{supp}}

\def\sign{\text{sign}\,}
\newcommand{\p}{\partial}
\newcommand{\bb}{\begin{equation}}
\newcommand{\ee}{\end{equation}}
\newcommand{\ba}{\begin{array}}
\newcommand{\ea}{\end{array}}
\newcommand{\f}{\frac}
\usepackage[all]{xy}
\newcommand{\ds}{\displaystyle}

\newcommand{\al}{\alpha}

\numberwithin{equation}{section}

\title{The Cauchy problem and continuation of periodic solutions for a generalized Camassa-Holm equation}

\author[1]{Nilay Duruk Mutlubas}  \author[2]{Igor Leite Freire}
\affil[1]{Faculty of Engineering and Natural Sciences, Sabanci University, Turkey
\texttt{nilay.duruk@sabanciuniv.edu}}

\affil[2]{Centro de Matem\'atica, Computa\c{c}\~ao e Cogni\c c\~ao,
Universidade Federal do ABC, Avenida dos Estados, $5001$, Bairro Bangu,
$09.210-580$, Santo Andr\'e, SP - Brasil\\
  \texttt{igor.freire@ufabc.edu.br} \\
  \texttt{igor.leite.freire@gmail.com}}

\begin{document}
\maketitle
\begin{abstract}
\centering\begin{minipage}{\dimexpr\paperwidth-9cm}
\textbf{Abstract:} We consider a three-parameter family of non-linear equations with $(p+1)-$order non-linearities. Such family includes as a particular member the well-known $b-$equation, which encloses the famous Camassa-Holm equation. For certain choices of the parameters we establish a global existence result and show a scenario that prevent the wave breaking of solutions. Also, we explore unique continuation properties for some values of the parameters.

\vspace{0.1cm}
\textbf{2020 AMS Mathematics Classification numbers}: 35A01, 35B60.


\end{minipage}
\end{abstract}
\bigskip
\newpage
\tableofcontents
\newpage

\section{Introduction}\label{sec1}

Our work is concerned with qualitative properties of periodic solutions of the equation
\bb\label{1.0.1}
u_t-u_{txx}+(b+c)u^pu_{x}=bu^{p-1}u_{x}u_{xx}+cu^pu_{xxx}.
\ee
Henceforth we assume that the parameters $b$, $c$ and $p$ are non-negative, positive real, and positive integer numbers,  respectively. 

In the case $p=c=1$ we have the $b-$equation \cite{dehoho}
\bb\label{1.0.2}
u_t-u_{txx}+(b+1)uu_{x}=bu_{x}u_{xx}+uu_{xxx},
\ee
whose relevance in hydrodynamics is discussed in \cite{dgh1,dgh2}, see also \cite{hs-siam,hs-pla} for a wide investigation about properties of these equations.
    
It is worth mentioning that the Degasperis-Procesi (DP) \cite{dp} and the Camassa-Holm (CH) \cite{chprl} equations can be recovered from the $b-$ equation when $b=2$ and $b=3$, respectively.
    
Whenever $b=(p+1)$ and $c=1$ we have the equation
$$
u_t-u_{txx}+(p+2)u^pu_{x}=(p+1)u^{p-1}u_{x}u_{xx}+u^pu_{xxx},
$$
which was obtained in \cite{silva-freire2015} using symmetries and conservation laws arguments, as well as in 
\cite{anco} where it can be inferred from a four-parameter family of equation as a member having the $H^1(\R)$ norm of the solutions (vanishing as $x\rightarrow\pm\infty$) conserved, peakons, and multi-peakon solutions. It can also be obtained from the family of equations studied in \cite{himonas-ade}. The two most famous members of this family is the already mentioned CH equation, for $p=1$, and the Novikov equation \cite{nov} when $p=2$. For further details, see the review \cite{igor-cm} where such equation is discussed in detail.

Our main motivation for this work is some open problems pointed out in \cite{himonas-jmp-2014}. We extend results on global existence of solutions established in this reference for the periodic problem associated to \eqref{1.0.1}, for certain parameters. Moreover, we study the prevention of blow up of periodic solutions of \eqref{1.0.1}, and establish unique continuation results.

Our paper is structured as follows: in Section \ref{sec2}, we present the notation and conventions used in the manuscript, as well as our main motivation, a revision of literature and our main contributions. Next, in Section \ref{sec3}, we prove some technical results very useful for establishing our achievements.

Global existence of solutions are proved in Section \ref{sec4}, whereas the prevention of their blow-up is described in Section \ref{sec5}. Such results, however, are for $b=0$ in \eqref{1.0.2}. We observe that the value $b=0$ in \eqref{1.0.2} brings some difficulty to the equations, as can be inferred from \cite{anco}, where very few conserved quantities can be found whenever $b=0$ in \eqref{1.0.1}, at least of low order, see also \cite{hs-siam,zhou}. Moreover, the results proved in \cite{hakkaev,himonas-jmp-2014,zhou} dealing with (global) existence of solutions usually assume that $b>0$. Part of the results reported in our work fills such gap (see theorems \ref{thm2.1} and \ref{thm2.2}), where we use a remarkable functional (whose square root is equivalent to the $H^3(\s)-$norm of the solutions of \eqref{1.0.2} with $b=0$). From this functional, we establish conditions for the existence of global solutions for the equation. 

Finally, unique continuation properties of \eqref{1.0.1} are presented in Section \ref{sec6}. Such results are dependent on the values of the parameters, but we do not need to fix a specific order on the non-linearities of the equation. 

Last but not least, our results are discussed in Section \ref{sec7} and our conclusions are reported in Section \ref{sec8}.

\section{Notation, conventions and main results}\label{sec2}

In this section, we fix the notations and conventions used throughout the manuscript, as well as our main contributions.

\subsection{Notation and conventions}

We identify the interval $[0,1)$ with the circle $\s$. The Sobolev space of order $s$ over the circle $\s$ and its corresponding norm are denoted by $H^s(\s)$ and $\|\cdot\|_{H^s(\s)}$, respectively. The functions $u$ considered in this work are usually dependent on the variables $t$ and $x$, that are regarded as time and space variables, respectively. Derivatives of $u$ with respect to its first argument will usually be denoted either by $u_t$ or $\p_t u$, while derivatives with respect to second argument will be denoted either by $u_x$ or $\p_x u$, respectively. For each $t\in[0,T)$, the quantity $m(t,x):=(1-\p^2_x)u(t,x)$ is sometimes referred as {\it momentum}. In particular, at $t=0$ we have $u(0,x)=u_0(x)$ and $m_0=u_0-u_0''$. Also, we assume that $u$ is at least $C^1$ with respect to its first argument.

By $\Lambda^{2}:H^s(\s)\rightarrow H^{s-2}(\s)$ we denote the  {\it Helmholtz operator} $1-\p_x^2$, whose inverse $\Lambda^{-2}$, acting on a function $f$, is given by $\Lambda^{-2}f=g\ast f$, where $\ast$ denotes the convolution between functions, and
\bb\label{g}
g(x)=\frac{\cosh(x-\lfloor x \rfloor - 1/2)}{2\sinh(1/2)}.
\ee

Note that $g(\cdot)$ is differentiable {\it a.e.} and, in particular, as soon as its derivative is defined, we have
\bb\label{2.1.2} 
\p_xg(x)=\frac{\sinh(x-\lfloor x \rfloor - 1/2)}{2\sinh(1/2)}.
\ee
Moreover, we have $\p_x\Lambda^{-2}(f)=(\p_x g)\ast f$.

\subsection{Motivation of our work and main results}

We recall some results appearing in the literature about \eqref{1.0.1} regarding the existence of its solutions. We observe that most of them deal with non-periodic problems and rely on the assumptions on the constants $b$, $c$, $p$ in (\ref{1.0.1}), initial data, and the order of Sobolev space. 

The Cauchy problem associated with (\ref{1.0.1}) was investigated in \cite{himonas-ade} in details. Using a Galerkin-type approximation scheme, local well-posedness in $H^s$, for $s>3/2$, has been proved both on the real line and the circle, i.e., given (\ref{1.0.1}) subject to $u(0,x)=u_0(x)\in H^s$, $s>3/2$, there exists a unique solution $u(t)\in C^0([0,T),H^s)\cap C^1([0,T),H^{s-1})$ which depends continuously on the initial data \cite[Theorem 1]{himonas-ade}, see also \cite[Theorem 1.1]{himonas-jmp-2014}. Moreover, it is proved that the solution map is not uniformly continuous neither on the line nor on the circle \cite[Theorem 2]{himonas-ade}. 

More recently, in \cite[Theorem 2.1]{yan}, local existence of solution for the non-periodic case, was proved in Besov spaces.

Global solutions for \eqref{1.0.1} subject to $u(0,x)=u_0(x)$ are granted provided that:
\begin{itemize}
    \item $p=1$ and $u_0(x)\in H^s(\mathbb{R})\cap L^1(\mathbb{R})$, $s> 3/2$, having the property that $m_0=u_0-(u_0)_{xx}$ does not change sign \cite[Theorem 3]{lai}, or
    \item $b\in [0,p/2]$, $u_0(x)\in H^s(\mathbb{R})$ with $s> 5/2$\cite[Theorem 1.5]{wei}, or
    \item $b=p/2$, $u_0(x)\in H^s(\mathbb{R})$ with $s> 3/2$\cite[Theorem 1.6]{wei}, or
    \item $b\in (0,p)$, $u_0(x)\in H^s(\mathbb{R})\cap W^{2,p/b}$ with $s> 3/2$\cite[Theorem 1.6]{wei}, or
    \item $b=p+1$, $u_0(x)\in H^s(\mathbb{R})$ with $s> 3/2$ having the property that $m_0=u_0-(u_0)_{xx}$ does not change sign \cite[Theorem 1.6]{wei}, or
    \item $b=p$ or $b\in\mathbb{R}$ and $p=1$ having the property that $m_0=u_0-(u_0)_{xx}$ does not change sign, $u_0(x)\in H^s(\mathbb{R})$ with $s> 3/2$  \cite[Theorem 1.6]{wei}, or
    \item $c=1$, $b=p+1$, $p$ odd, $u_0(x)\in H^s(\mathbb{R})$ with $s> 3/2$, and supposing the existence of $x_0\in\R$ such that $m_0(x)\leq x$ if $x\leq x_0$, and $m_0(x)\geq x$ if $x\geq x_0$, see \cite[Theorem 1.1]{chen}, or

    \item $p=1$, $c=1$, $b\geq1$, and $u_0(x)\in H^s(\s)$, $s\geq 3$, having the property that $m_0=u_0-(u_0)_{xx}$ does not change sign \cite[Theorem 4.1]{hakkaev}.
\end{itemize}

We observe that apart from the work by Christov and Hakkaev \cite{hakkaev}, the other papers listed above are all concerned with the non-periodic case.

It is worth mentioning that Himonas and Thompson \cite{himonas-jmp-2014} state that the existence of global solutions for \eqref{1.0.1} subject to $u(0,\cdot)=u_0(\cdot)\in H^s$, $s>3/2$, (both periodic and non-periodic cases), for $b\neq p+1$ is an open problem, see \cite[page 3]{himonas-jmp-2014}. This observation is one of the main motivations for our work. We shed light to this question through the following result concerning the case $b=0$ and $p=1$.
\begin{theorem}\label{thm2.1}
Let $u_0\in H^3(\s)$ and $u(t,.)$ be the unique local solution of 
\begin{equation} \label{2.2.1} 
\left\{
\begin{array}{l}
u_t-u_{txx}+cuu_{x}=cuu_{xxx},~~x\in\mathbb{R},~~t>0,\\
\\
u(0,x)=u_0(x), ~~x\in\mathbb{R},\\
\\
u(t,x)=u(t,x+1),~~x\in\mathbb{R},~~t>0.
\end{array}\right.
\end{equation}
If $m_0=u_0-u_0''\in H^1(\s)$ does not change sign, then the solution exists globally in \mbox{$C^0([0,\infty),H^3(\s))\cap C^1([0,\infty),H^2(\s))$}.
\end{theorem}

Note that our Theorem \ref{thm2.1}, jointly with \cite[Theorem 2.4]{silva-freire2020} extend \cite[Theorem 1.4]{himonas-jmp-2014} to the equation \eqref{2.2.1}, with $u_0\in H^3(\s)$, for both periodic and non-periodic cases.

Our main result regarding blow-up of the solutions of \eqref{2.2.1} is:

\begin{theorem}\label{thm2.2}
Let $u_0\in H^3(\s)$ and $u\in C^0([0,\infty),H^3(\s))\cap C^1([0,\infty),H^2(\s))$ be the corresponding solution to \eqref{2.2.1}. If there exists a constant $M>0$ such that 
\begin{displaymath}
||u||_{L^\infty(\s)}+||u_x||_{L^\infty(\s)}<M,
\end{displaymath}
then $||u||_{H^3(\s)}$ does not blow up in finite time. 
\end{theorem}

In \cite[Theorem 1.2]{himonas-jmp-2014} it was shown that if the initial data and its derivative are $o(e^{-x})$ and $O(e^{-\al x})$, respectively,  as $x\rightarrow\infty$, for a suitable $\al\in(0,1)$, and if there exists a $t_1\in(0,T)$ such that the corresponding non-periodic solution satisfies
$$
\lim_{x\rightarrow\infty}\f{|u(t_1,x)|}{e^{x}}=0,
$$
then the solution $u$ vanishes identically (see \cite[Theorem 1.2]{himonas-jmp-2014} for the restriction on the parameters). Such unique continuation result cannot be proved for \eqref{1.0.1} using the same ideas in \cite{himonas-jmp-2014} for the periodic case:  their proof makes use of the behaviour of the solutions for $|x|\gg 1$. This does not imply the impossibility of establishing an analogous result for \eqref{1.0.1}. Actually, our next two results can be seen as the periodic counter-parts of \cite[Theorem 1.2]{himonas-jmp-2014}.

\begin{theorem}\label{thm2.3}
Assume that $u\in C^0([0,T),H^s(\R))$, $s>3/2$, is a solution of \eqref{1.0.1}. Suppose that
\begin{itemize}
    \item $p=1$, $b\in[0,3c]$, and $s>3/2$; or
    \item $b=pc$, $p$ is odd, and $s>3/2$; or
    \item $b=pc$, $p$ is even, $m_0$ is either non-negative or non-positive, and $s=3$.
\end{itemize}
If there exists $t^\ast\in(0,T)$, $a,b\in\s$, with $a<b$, such that $u(t^\ast,\cdot)\big|_{[a,b]}\equiv0$ and $u_t(t^\ast,a)=u_t(t^\ast,b)=0$, then $u\equiv0$.
\end{theorem}

In line with the above result, we have the following unique continuation property.

\begin{theorem}\label{thm2.4}
Assume that $u\in C^0([0,T),H^s(\R))$, $s>3/2$, is a non-negative or non-positive solution of \eqref{1.0.1}, and suppose that $p=1$ and $b\in[0,3c]$, or $b=pc$. If there exists a piece of cylinder ${\cal S}=(t_0,t_1)\times (x_0,x_1)\subseteq[0,T)\times\s$ such that $u\big|_{\cal S}\equiv0$, then $u\equiv0$.
\end{theorem}

\section{Preliminaries}\label{sec3}

In this section we present some technical and necessary results needed for the proof of our main theorems. It will be very useful to rewrite \eqref{1.0.1} in terms of its momentum $m=u-u_{xx}$, that is,
\bb\label{3.0.1}
m_t+cu^pm_x+bu^{p-1}u_xm=0.
\ee

Note that in view of the relation $m(t,x)=\Lambda^2 u(t,x)$, the solution of \eqref{1.0.1} subject to $u(0,x)=u_0$ is equivalent to \eqref{3.0.1} subject to $m(0,x)=m_0(x)$.

\begin{theorem}\label{thm3.1}
Let $u_0\in H^3(\s)$, $u_0\not\equiv0$, and $u\in C^0([0,T),H^3(\s))\cap C^1([0,T),H^2(\s))$, be the corresponding solution to \eqref{1.0.1} subject to $u(0,x)=u_0(x)$. If $m_0(x)\geq0$ or $m_0(x)\leq0$, $x\in\R$, then the corresponding momentum $m(t,x)$ satisfies $m(t,x)\geq0$ or $m(t,x)\leq0$, respectively.
\end{theorem}

\begin{proof}Our demonstration follows the same steps as \cite[Theorem 3.1]{const} and we follow it closely, see also \cite[Theorem 4.1]{hakkaev}. Let us fix $x$ and define $\gamma(\cdot,x)$ by
$$
\left\{
\ba{lcl}
\gamma_t(t,x)&=&c u(t,\gamma)^p,\quad t\in(0,T),\\
\\
\gamma(0,x)&=&x.
\ea
\right.
$$

Fixing $t$, letting $x$ varies, and differentiating the PVI above with respect to $x$, we obtain
\bb\label{3.0.2}
\left\{
\ba{lcl}
\gamma_{tx}(t,x)&=&p c (u^{p-1}u_x\gamma_x)(t,\gamma),\quad t\in(0,T),\\
\\
\gamma_x(0,x)&=&1.
\ea
\right.
\ee

From \eqref{3.0.2} we obtain
\bb\label{3.0.3}
\gamma_x(t,x)=e^{\int_0^t p c (u^{p-1}u_x)(\tau,\gamma(\tau,x))d\tau}>0.
\ee
Proceeding similarly as \cite[Theorem 3.1]{const} we conclude that $\gamma(t,\cdot):\R\rightarrow\R$ is an increasing diffeomorphism for every $t\in[0,T)$. Moreover, since $\gamma_t=cu^p$ and $\gamma_{tx}=p c (u^{p-1}u_x\gamma_x)$, we have
$$
\ba{lcl}
\ds{\f{d}{dt}\Big(m(t,\gamma)\gamma_x(t,y)^{\f{b}{pc}}\Big)}&=&\ds{\Big[\Big(m_t+m_x\gamma_t\Big)\gamma_x^{\f{b}{pc}}+\f{b}{pc}m\gamma_x^{\f{b}{pc}-1}\Big](t,\gamma)}\\
\\
&=&\ds{\Big[\Big(m_t+\gamma_tm_x+bu^{p-1}u_xm \Big)\gamma_x^{\f{b}{pc}}\Big](t,\gamma)}\\
\\
&=&\ds{\Big[\Big(m_t+cu^pm_x+bu^{p-1}u_xm\Big)\gamma_x^{\f{b}{pc}}\Big](t,\gamma)=0}.
\ea
$$

The differential equation above implies that $m(t,\gamma(t,x))\gamma_x^{\f{b}{pc}}=m_0(x)$. Therefore, if $m_0(x)\geq0$ or $m_0(x)\leq0$, then $m(t,y(t,x))\geq0$ or $m(t,y(t,x))\leq0$, respectively, since \eqref{3.0.3} holds.
\end{proof}

\begin{corollary}\label{cor3.1}
Let $0\not\equiv u_0\in H^3(\s)$ be a non-trivial function and $u\in C^0([0,T),H^3(\s))\cap C^1([0,T),H^{2}(\s))$, be the corresponding solution of \eqref{1.0.1} subject to $u(0,x)=u_0(x)$. If $m_0(x)\geq0$ or $m_0(x)\leq0$, for every $x\in\R$, then $u(t,x)\geq0$ or $u(t,x)\leq0$, respectively. Moreover, if $m_0(x)$ does not change sign, then $\sign{(u)}=\sign{(u_0)}=\sign{(m_0)}=\sign{(m)}$. 
\end{corollary}

\begin{proof}

Under the conditions above, Theorem \ref{thm3.1} implies that $x\mapsto m(t,x)$ is non-negative or non-positive, for all $t\in[0,T)$. Since
$$
u(t,x)=g\ast m(t,x)=\int_\s g(x-y)m(t,y)dy,
$$
and $g$, given by \eqref{g}, is strictly positive, then the same will apply to $u$, respectively.
\end{proof}

Equation \eqref{1.0.1} has several conserved quantities depending on the values of the parameters $b$, $c$ and $p$, see \cite[Theorem 2.1]{anco}. Among them, two are very important for our purposes, namely,
\bb\label{3.0.4}
{\cal H}_1(t)=\int_\s u(t,x)dx,
\ee
as long as $p=1$ or $b=pc$, and
\bb\label{3.0.5}
{\cal H}_2(t)=\f{1}{2}\int_\s (u^2+u_x^2)dx,
\ee
in case $b=(p+1)c$.

If $u$ is a solution of \eqref{1.0.1}, then from \cite[Theorem 2.1]{anco} we can obtain the relation
$$
\p_t(u-u_{xx})=-\p_x\Big(\f{c}{p+1}u^{p+1}+\f{pc-b}{2}u_x^2-cu^pu_{xx}\Big),
$$
provided that $p=1$ or $b=pc$. Therefore 
$$
\f{d}{dt}\int_\s m dx=\int_\s m_t dx= -\int_\s \p_x\Big(\f{c}{p+1}u^{p+1}+\f{pc-b}{2}u_x^2-cu^pu_{xx}\Big)dx=0,
$$
which implies the conservation of the momentum, that is, 
\begin{equation}\label{3.0.6}
\int_\s m(x) dx=\int_\s m_0(x)dx.
\end{equation}

The main relevance of theorem \ref{thm3.1}, or its corollary, is that it assures the existence of non-negative/non-positive solutions with certain regularity. From this and the conserved quantities we have the following result.

\begin{theorem}\label{thm3.2}
Let $u_0\in H^3(\s)$, $p=1$ or $b=pc$, and assume that $u$ is either non-negative or non-positive. Then the $L^1(\s)-$norm of the corresponding solution $u\in C^0([0,T),H^3(\R))\cap C^1([0,T),H^{2}(\s))$ of \eqref{1.0.1} subject to $u(0,x)=u_0(x)$ is conserved.  In particular, if $m_0(x)$ is either non-positive or non-negative for every $x\in\R$, then the corresponding solution conserves the $L^1(\s)-$norm.
\end{theorem}

\begin{proof}
Assume that $u(t,x)\geq0$. The conserved quantity \eqref{3.0.4} yields ${\cal H}_1(t)=\|u(t,\cdot)\|_{L^1(\s)}$. On the other hand, if $u(t,x)\leq0$, then $-u(t,x)\geq 0$ and $\|u(t,\cdot)\|_{L^1(\s)}=-{\cal H}_1(t)$.
\end{proof}

\begin{corollary}\label{cor3.2}
Let $p=1$ or $b=pc$. If $m_0\in H^1(\s)$ is non-negative or non-positive, then the corresponding solution $m(t,\cdot)$ of the equation \eqref{3.0.6} belongs to $L^1(\s)$, for all $t\in[0,T)$.
\end{corollary}

\begin{proof}
First we note that \eqref{3.0.6} holds under the conditions in the Corollary. By Theorem \ref{thm3.1} we see that $\sign{(m(t,x))}=\sign{(m_0(x))}$. The result follows noticing that, up eventually a sign, \eqref{3.0.4} is the $L^1(\s)-$norm of the momentum.
\end{proof}

\section{Global solution and the proof of theorems \ref{thm2.1} and \ref{thm2.2}} \label{sec4}

Through this section we assume that $p=1$ and $b=0$ in \eqref{1.0.1}. This fact is assumed henceforth without further mention.

We begin observing that 
$$\int_{\s}u_{xx}dx=0.$$
This fact is enough to guarantee the existence, for each $t\in(0,T)$, the existence of a point $\xi_t-1\in(0,1)$ such that $u_x(t,\xi_t-1)=0$.

\begin{lemma}\label{lem4.1}
If $u_0\in H^3(\s)\cap L^1(\s)$, is such that $m_0\geq 0$ (or $m_0\leq 0$), then there exists a constant $K>0$ such that the solution of \eqref{2.2.1} satisfies $||u_x||_{L^\infty(\s)}\leq K$.
\end{lemma}

\begin{proof}
Assume $m_0$ does not change sign and $m_0\geq 0$. Then, by \eqref{3.0.6}
\begin{eqnarray*}
C_1=||m_0||_{L^1(\s)}&=&\int_\s m_0(r)dr=\int_\s m(t,r)dr\nonumber\\
&=&\int_{\xi_{t}-1}^{\xi_t}m(t,r)dr \geq\int_{{\xi_t}-1}^x m(t,r)dr\nonumber\\
&=&\int_{\xi_{t}-1}^x (u-u_{xx})(t,r)dr\nonumber\\
&=&\int_{\xi_{t}-1}^x u(t,r)dr-u_x(t,x)
\end{eqnarray*}
holds for every $x\in[\xi_{t}-1, \xi_t]$. Hence,
\begin{eqnarray*}
-u_x(t,x)\leq C_1-\int_{\xi_{t}-1}^x u(t,r)dr\leq C_1+\int_{\xi_{t}-1}^x
u(t,r)dr\leq C_2
\end{eqnarray*}
for some positive constant $C_2$. Here, we use (\ref{3.0.4}) and Theorem \ref{thm3.1}, which guarantee that $u$ does not change sign, under the assumption that $m_0$ does not change sign. Taking into account the final result, we observe that $u_x$ is bounded from below.

Moreover, the existence of a point $\nu_t\in (0,1)$ such that $u_x(t,\nu_t)=0$, implies for every \mbox{$x\in[\nu_{t}-1, \nu_t]$} that
\begin{eqnarray*}
u_x(t,x)+\int_x^{\nu_t}u(t,r) dr&=&\int_x^{\nu_t}(u-u_{xx})(t,r) dr\\
&\leq&\int_{\nu_{t}-1}^{\nu_t}(u-u_{xx})(t,r) dr=C_1
\end{eqnarray*}
and 
\begin{eqnarray*}
u_x(t,x)\leq C_1-\int_x^{\nu_t}u(t,r) dr= C_1+\int_x^{\nu_t}u(t,r) dr\leq
C_2.
\end{eqnarray*}
Hence, $u_x$ is bounded also from above. Therefore, we can conclude
that $||u_x||_{\infty}$ norm is bounded provided that $m$ does not change sign , i.e. $||u_x||_{\infty}\leq K$. The case $m_0\leq0$ is proved in a similar way and, therefore, is omitted.
\end{proof}

To prove Theorem \ref{thm2.1}, we need the following lemma which ensures the boundedness of Sobolev norm of the solutions of \eqref{1.0.1}:

\begin{lemma}\label{lem4.2}
Let $u_0\in H^3(\s)$ and $u(t,.)$ be the corresponding unique local solution of \eqref{2.2.1}. If Lemma \ref{lem4.1} is verified, then $||u||_{H^3(\s)}\leq  e^{cAt}||u_0||_{H^3(\s)}$, for some constant $A$ depending on $||u_0||_{L^1(\s)}$. 
\end{lemma}

\begin{proof} We proceed similarly as in \cite{silva-freire2020} for establishing an analogous result for the non-periodic case. We give the steps for finding the form of a functional $I(u)$ which will be used to estimate $H^3$-norm of $u(t,x)$:
\begin{itemize}
    \item Multiply \eqref{2.2.1} by $u$ and rewrite to get 
\begin{equation}\label{4.0.1}
\frac{d}{dt}\left(\frac{u^2+u_x^2}{2}\right)+\frac{d}{dx}\left(\frac{c}{3}u^3-cu^2u_{xx}-uu_{xt}\right)+cu(u_x^2)_x=0.
\end{equation}
    \item Differentiate \eqref{2.2.1} with respect to $x$, multiply by $u_x$ and rewrite to get
\begin{equation}\label{4.0.2}
\frac{d}{dt}\left(\frac{u_x^2+u_{xx}^2}{2}\right)-\frac{d}{dx}(u_xu_{xxt}+cuu_xu_{xxx})+\frac{c}{2}u(u_x^2+u_{xx}^2)_x+cu_x^3=0.
\end{equation}
    \item Differentiate \eqref{2.2.1} with respect to $x$ twice, multiply by
$u_{xx}$ and rewrite to get

\begin{eqnarray} \label{4.0.3}
&&\frac{d}{dt}\left(\frac{u_{xx}^2+u_{xxx}^2}{2}\right)-\frac{d}{dx}(u_{xx}u_{xxxt}+cuu_{xx}u_{xxxx}+cu_xu_{xx}u_{xxx})\nonumber\\
&&+\frac{c}{2}u(u_{xx}^2+u_{xxx}^2)_x+3cu_xu_{xx}^2+cu_xu_{xxx}^2=0.
\end{eqnarray}
\end{itemize}

Defining
\bb\label{4.0.4}I(u)=\int_\s\left(\frac{u^2+u_x^2}{4}+\frac{u_x^2+u_{xx}^2}{2}+\frac{u_{xx}^2+u_{xxx}^2}{2}\right)dx,\ee
from \eqref{4.0.1}--\eqref{4.0.3} and Lemma \ref{lem4.1}, we have,

\begin{eqnarray*}
\frac{d}{dt}I(u)&=&-2c\int_\s u_xu_{xx}^2dx-\frac{c}{2}\int_\s
u_xu_{xxx}^2dx\\
&=&c\int_\s (-u_x)\left(2u_{xx}^2+\frac{u_{xxx}^2}{2}\right)dx\\
&\leq&cC_2\int_\s\left(2u_{xx}^2+\frac{u_{xxx}^2}{2}\right)dx\\
&\leq&cA I(u),
\end{eqnarray*}
for some optimal constant $A>0$. By Gronwall's inequality, 
\begin{equation*}
I(u)\leq I_0e^{cAt},
\end{equation*}
for some constant $I_0$ depending on
$||u_0||_{H^3(\s)}$. We conclude the result noticing that $\sqrt{I(u)}$ is a norm equivalent to $\|u\|_{H^3(\s)}$.
\end{proof}

\textbf{Proof of Theorem \ref{thm2.1}}
Global existence of solution of \eqref{2.2.1} is a direct consequence of Lemma \ref{lem4.2} since the solution can be estimated at any finite time, therefore can be extended globally in time.

\section{Blow-up}\label{sec5}

The following lemma, provided that Lemma \ref{lem4.1} holds, will imply in Theorem \ref{thm2.2} that the solution $u(t,.)$ does not blow up in finite time under the given assumptions:

\begin{lemma}\label{lem5.1}
If  $T<\infty$, then $\sup_{t\in[0,T)}||u(t,\cdot)||_{L^\infty(\s)}<\infty$, where $u$ is the solution of \eqref{2.2.1} in $C^0([0,T),H^3(\s))\cap C^1([0,T),H^{2}(\s))$.
\end{lemma}

\begin{proof}
let $y\in\s$ and $\tau\in[0,T)$ be arbitrary. Consider the
characteristic ordinary differential equation with initial condition
for a curve $\gamma:[0,T)\rightarrow \s$ corresponding to the
solution $u$, that is,
\begin{eqnarray*}
&\dot{\gamma}(r)=cu(r,\gamma(r)),\\
&\gamma(\tau)=y.
\end{eqnarray*}
Here $u\in C^0([0,T),H^3(\s))\subseteq C^0([0,T),C^1(\s))$ by Sobolev
embedding theorem. Therefore, $(t,x)\mapsto u(t,x)$ is continuous on
$[0,T)\times \s$ and globally Lipschitz with respect to $x$. By
compactness of $\s$, we get a unique global solution $\gamma\in
C^1([0,T),\s)$ to the characteristic initial value problem given
above. That follows,
\begin{eqnarray*}
\frac{d}{dr}(u(r,\gamma(r)))&=&u_t(r,\gamma(r))+u_x(r,\gamma(r))\dot{\gamma}(r)\\
&=&u_t(r,\gamma(r))+cu_x(r,\gamma(r))u(r,\gamma(r))\\
&=&-\frac{3c}{2}\partial_x\Lambda^{-2}(u_x^2),
\end{eqnarray*}
and
\begin{eqnarray*}
u(\tau,y)=u(\tau,\gamma(\tau))&=&u(0,\gamma(0))-\frac{3c}{2}\int_0^{\tau}(\partial_x\Lambda^{-2}(u_x^2))(r,\gamma(r))dr\\
&=&u_0(\gamma(0))-\frac{3c}{2}\int_0^{\tau}(\partial_x\Lambda^{-2}(u_x^2))(r,\gamma(r))dr.
\end{eqnarray*}

Therefore,
\begin{eqnarray}\label{5.0.1}
|u(\tau,y)|\leq
||u_0||_{\infty}+\frac{3c}{2}\int_0^{\tau}|(\partial_x\Lambda^{-2}(u_x^2))(r,\gamma(r))|dr.
\end{eqnarray}

The following estimation can be done for the integrand by periodicity and Lemma \ref{lem4.1}:
\begin{eqnarray*}
|(\partial_x\Lambda^{-2}(u_x^2))(r,\gamma(r))|&\leq&
||\partial_x\Lambda^{-2}(u_x^2)||_{\infty}\leq
||\partial_x\Lambda^{-2}(u_x^2)||_{H^1(\s)}\\&\leq&
||u_x^2||_{L^2(\s)}\leq ||u_x^2||_{L^\infty(\s)}\leq K^2.
\end{eqnarray*}

Using this in (\ref{5.0.1}) and observing that $y$ and $\tau$ are
arbitrary, the upper bound for $|u(\tau,y)|$ is independent of these
variables, we can conclude that if $T<\infty$, then
$\sup_{t\in[0,T)}||u(t)||_{L^{\infty}(\s)}<\infty$.
\end{proof}
We now prove our main result regarding blow-up.

{\bf Proof of Theorem \ref{thm2.2}}
The theorem can be directly observed from Lemma \ref{lem4.2}. Moreover, recalling the results of Lemma \ref{lem4.1} and Lemma \ref{lem5.1}, one can also observe that there is no wave breaking in finite time since the wave breaks only if $||u_x||_{L^\infty(\s)}$ is unbounded whereas $||u||_{L^\infty(\s)}$ is bounded.
\hfill$\square$

\section{Unique continuation results}\label{sec6}

Given a solution $u\in C^0([0,T),H^s(\s))$, with $s>3/2$, of \eqref{1.0.1}, for each $t\in[0,T)$ fixed, we define
\bb\label{6.0.1}
f_t(x):=\f{3pc-b}{2}u(t,x)^{p-1}u_x(t,x)^2+\f{b}{p+1}u(t,x)^{p+1}
\ee
and
\bb\label{6.0.2}
F_t(x)=\p_x\Lambda^{-2}f_t(x).
\ee

Finally, consider the families
\bb\label{6.0.3}
{\cal F}_1=(f_t)_{t\in[0,T)}\quad\text{and}\quad {\cal F}_2=(F_t)_{t\in[0,T)}.
\ee
In view of our assumptions, ${\cal F}_1$ is a family of continuous and bounded functions, whereas members of ${\cal F}_2$ are smooth and bounded.

Throughout this section, very often we will make use of the following hypothesis (they are conditions in theorems \ref{thm2.3} and \ref{thm2.4}):
\begin{itemize}
\item[{\bf H1}]\label{h1} $p=1$ and $b\in[0,3c]$;
\item[{\bf H2}]\label{h2} $b=pc$.
\end{itemize}

\begin{proposition}\label{prop6.1}
Let $u_0\in H^s(\s)$, $u\in C^0([0,T),H^s(\s))\cap C^1([0,T),H^{s-1}(\s))$ be the corresponding solution of \eqref{1.0.1} subject to $u(0,x)=u_0(x)$, and $f_t(x)$ given by \eqref{6.0.1}.
\begin{enumerate}
\item If $s>3/2$ and {\bf H1} holds, then $f_t(x)\geq0$.
\item If $s>3/2$, {\bf H2} holds and $p$ is odd, then $f_t(x)\geq0$.
\item If $s=3$, {\bf H2} holds, $p$ is even and $m_0(x)$ is either non-negative or non-positive, then $f_t(x)\geq0$ or $f_t(x)\leq0$, respectively. 
\end{enumerate}

Moreover, under the conditions above, $u(t,\cdot)=0\Leftrightarrow f_t(x)=0 \Leftrightarrow \Lambda^{-2}f_t(x)=0$.
\end{proposition}

\begin{proof}
If $s>3/2$ and {\bf H1} holds, then 
$$f_t(x)=\f{3c-b}{2}u_x(t,x)^2+\f{b}{2}u(t,x)^2$$
and the result is immediate.

Let us now assume that {\bf H2} holds. The quantities $u^{p-1}u_x^2$ and $u^{p+1}$ are non-negative as long as $p$ is odd, whereas for an even $p$, $f_t$ will be non-negative or non-positive provided that $u$ is non-negative or non-positive, respectively. In any circumstance, $f_t(\cdot)=0$ if and only if $u(t,\cdot)=0$. We now observe that $\Lambda^{-2} f_t(x)=g\ast f_t(x)$, where $g$ is given by \eqref{g} and is strictly positive. Therefore, $\Lambda^{-2} f_t(x)=0$ if and only if $f_t(\cdot)=0$.
\end{proof}

\begin{corollary}\label{cor6.1}
Under the conditions in Proposition \ref{prop6.1}, if $u(t^\ast,\cdot)=0$, for some $t^\ast\in[0,T)$, then $u(t,\cdot)=0$ for every $t\in[0,T)$. In particular, the families ${\cal F}_1$ and ${\cal F}_2$, given by \eqref{6.0.3} reduce to the identically vanishing function.
\end{corollary}

Note that corollary \ref{cor6.1} holds for any conservative solution, in the sense that if the solution $u$ conserves $\|u(t,\cdot)\|_X$ and $\|u(t_0,\cdot)\|_X=0$, for some $t_0$, then $\|u(t,\cdot)\|_X$ vanishes at all (and this does not necessarily requires local well-posedness of solutions, but only conservation of the norm).

\begin{proof}
Under the conditions above, we observe that $\|u(t,\cdot)\|_{X}$ is constant, where $X=H^1(\s)$ or $X=L^1(\s)$, meaning that for all $t\in[0,T)$, we have $\|u(t,\cdot)\|_{X}=\|u(t^\ast,\cdot)\|_{X}=0$.
\end{proof}

\begin{proposition}\label{prop6.2}
Let $u\in C^0([0,T),H^s(\s))\cap C^1([0,T),H^{s-1}(\s))$ be a solution of \eqref{1.0.1}, and assume that the conditions on $p$, $b$, $c$ and $s$ are as in Theorem \ref{thm2.3}. If there exists numbers $t^\ast$, $x_0$, $x_1$ such that $\{t^\ast\}\times[x_0,x_1]\subseteq(0,T)\times\s$, $f_{t^\ast}\big|_{[x_0,x_1]}\equiv0$ and $F_{t^\ast}(x_0)=F_{t^\ast}(x_1)$, then ${\cal F}_1={\cal F}_2=\{0\}$, where ${\cal F}_1$ and ${\cal F}_2$ are the families \eqref{6.0.3}.
\end{proposition}

\begin{proof} First of all, we note that under the conditions in Theorem \ref{thm2.3}, the function $u$ is either non-negative or non-positive. From the identity $\p_x^2\Lambda^{-2}=\Lambda^{-2}-1$, for every $x$ we have $F_{t^\ast}'(x)=\p_x F_{t^\ast}(x)=\Lambda^{-2}f_{t^\ast}(x)-f_{t^\ast}(x)$. Since $f_{t^\ast}(x)=0$ for $x\in[x_0,x_1]$, the Fundamental Theorem of Calculus and the hypothesis on $F_{t^\ast}$ and $f_{t^\ast}$ imply
\bb\label{6.0.4} 
0=F_{t^\ast}(x_1)-F_{t^\ast}(x_0)=\int_{x_0}^{x_1}\Lambda^{-2}f_{t^\ast}(x)dx.
\ee
Since $\Lambda^{-2}f_{t^\ast}(x)\geq0$ or $\Lambda^{-2}f_{t^\ast}(x)\leq0$, we are forced to conclude that $\Lambda^{-2}f_{t^\ast}(x)=0$ and again, by Proposition \ref{prop6.1}, we conclude that $f_{t^\ast}(x)=0$ and, as a consequence, $u(t^\ast,\cdot)=0$. The result follows from Corollary \ref{cor6.1}.
\end{proof}

{\bf Proof of Theorem \ref{thm2.3}.}
Assume that for some $t^\ast\in(0,T)$,  $u(t^\ast,\cdot)\big|_{[a,b]}\equiv0$, where $0<a<b<1$, then $f_{t^\ast}\big|_{[a,b]}\equiv0$. We observe that we can rewrite \eqref{1.0.1} in the non-local form
\bb\label{6.0.5} 
u_t+cu^pu_x=-\p_x\Lambda^{-2}\Big(\f{3pc-b}{2}u^{p-1}u_x^2+\f{b}{p+1}u^{p+1}\Big)-\f{(p-1)(b-pc)}{2}\Lambda^{-2}\big(u^{p-2}u_x^3\big).
\ee
Therefore, from \eqref{6.0.2}, \eqref{6.0.1} and \eqref{6.0.5}; and taking into account the restrictions on the parameters, we have
\bb\label{6.0.6} 
F_t(x)=-\Big(u_t+cu^pu_x\Big)(t,x).
\ee
Since $u_t(t^\ast,a)=u_t(t^\ast,b)=0$ and $u(t^\ast,x)=0,\,\,x\in[a,b]$, then $f_{t^\ast}(x)=0$, $a\leq x\leq b$ and $F_{t^\ast}(a)=F_{t^\ast}(b)=0$. By Proposition \ref{prop6.2} we conclude that $u(t,x)=0$, for every $(t,x)\in[0,T)\times\s$.
\hfill$\square$

We now give two different demonstrations for theorem \ref{thm2.4}: the first one uses the invariance of the $L^1(\R)$ norm of the solutions. The second one is based on the continuity of the data-to-solution map.

{\bf Proof of Theorem \ref{thm2.4}.}
\begin{itemize}
    \item {\bf Conserved quantity approach.} Under the conditions in the theorem, we can find $\{t^\ast\}\times[a,b]\subseteq{\cal S}$ such that $u(t^\ast,x)=0$, $a\leq x\leq b$ and $u_t(t^\ast,a)=u_t(t^\ast,b)=0$. Equation \eqref{6.0.6} implies that $F_{t^\ast}(a)=F_{t^\ast}(b)=0$, whereas equation \eqref{6.0.4} and Proposition \ref{prop6.2} tell us that $f_{t^\ast}(x)=0$, and then, $u\equiv0$.
    \item {\bf Local well-posedness approach.} By \eqref{6.0.4}, \eqref{6.0.6} and the fact that $f_{t^\ast}(\cdot)$ is either non-negative or non-positive, we conclude that $f_{t^\ast}(x)=0$, for all $x$. This implies that $u(t^\ast,\cdot)=0$. Since \eqref{1.0.1} is invariant under translations, we conclude that $v(t,x):=u(t+t^\ast,x)$ is also a solution of \eqref{1.0.1} subject to $v(0,x)=u(t^\ast,x)=0$, $x\in\mathbb{S}$. The data-to-solution continuity implies that $v\equiv0$ and $u\equiv0$ as well.
\end{itemize}

\hfill$\square$
 
We have a very simple, but beautiful, application of the results above. We recall that $u(t,x)$, $x\in\s$, is said to be compactly supported if $$\supp(u\big|_\s)=\overline{\{(t,x)\in[0,T)\times\s,\,\,u(t,x)\neq0\}}$$
is compact.

\begin{theorem}\label{thm6.1}
Let $u\in C^0([0,T),H^s(\s))$, $s>3/2$, be a solution of \eqref{1.0.1}. If $u\not\equiv0$, then it cannot be compactly supported.
\end{theorem}

\begin{proof}
Let us prove our result by contradiction. Suppose that $u$ is compactly supported. Let us firstly assume that $([0,T)\times\s)\setminus\supp(u\big|_\s)$ is connected. Therefore we can find an open set ${\cal S}$ contained on it. The result follows from Theorem \ref{thm2.4}.

In case $[0,T)\times\s\setminus\supp(u\big|_\s)$ is not connected, it then has some connected component, from which we can find a subset ${\cal S}$ such that the conditions in Theorem \ref{thm2.4} hold and the result is again a consequence of that theorem. 
\end{proof}

\section{Discussion}\label{sec7}

In \cite{himonas-jmp-2014} it was proved a global existence result for \eqref{1.0.1} under certain conditions on the parameters ($b=p+1$, in our notation) and the authors pointed out that the same problem for other values is an open question. Our Theorem \ref{thm2.1} covers a case not treated in the mentioned reference, since we consider the situation $b=0$ and $p=1$. In this sense, it extends to \eqref{2.2.1} the results established in \cite[Theorem 4.1]{himonas-jmp-2014}. Also, the way we proved it is completely different from the approach used in \cite{himonas-jmp-2014}. Therefore, not only we extended the result by Himonas and Thompson for a case not covered before, but the way we extended it is also different when compared with that used in the aforementioned work. In fact, our proof uses the identities \eqref{4.0.1}--\eqref{4.0.3} to estimate the functional \eqref{4.0.4}, which is equivalent to $\|\cdot\|_{H^3(\s)}$. Our approach is very similar to an analogous result for the non-periodic case firstly established in \cite{silva-freire2020}.

In line with our contribution mentioned above, we delineated conditions for preventing the manifestation of blow-up of the periodic solutions of \eqref{1.0.1} as a wave breaking in our Theorem \ref{thm2.2}, which is in line with similar results for equations of the Camassa-Holm type, see \cite[Theorem 3.4]{wei}.

We note that our theorems \ref{thm2.1} and \ref{thm2.2} fill some gaps in the qualitative analysis of solutions of the $0-$equation and, in particular, not only it partially answers some questions pointed out in \cite{himonas-jmp-2014}, but also extends several results mentioned in our Section \ref{sec2} to the $0-$equation. As pointed out by Zhou (see \cite[page 372]{zhou}), the description of necessary and sufficient conditions for the blow-up of the solutions of the $b-$equation is a difficult problem to be considered (and \eqref{2.2.1} in particular). Although Zhou's comments are concerned with the non-periodic case, it is somewhat expected that the periodic case inherits the same issues.

Finally, we also explored the problem of unique continuation of solutions of \eqref{1.0.1} in theorems \ref{thm2.3} and \ref{thm2.4}. Two key ingredient for dealing with it is the construction of the families \eqref{6.0.1} and \eqref{6.0.2}, and the fact that \eqref{6.0.1} is a non-negative/positive function, depending on the value of the parameters.

We give two demonstration for unique continuation of the solutions of \eqref{1.0.1} belonging to $C^0([0,T);H^s(\s))$, for $s>3/2$, under restriction on the parameters of the equation and $s$. The first demonstration is based on the invariance of a conserved quantity, namely, the $L^1(\s)-$norm. In fact, by proposition \ref{prop6.2} we can find a value of $t^\ast$ for which the solution vanishes identically. Due to this fact and the conservation of the $L^1(\s)-$norm of the solutions of \eqref{1.0.1} (see Theorem \ref{thm3.2}) we use the conserved quantities \eqref{3.0.4} and \eqref{3.0.5} to show that the solution vanishes for all time. This is based on the ideas introduced in \cite{igor-jpa,freire-cor} for investigation of unique continuation results through the use of invariants conserved along time.

The second proof uses the continuity of the data-to-solution map to show that the only solution of \eqref{1.0.1}, under the conditions in theorem \ref{thm2.4}, is zero. Such demonstration is based on the paper \cite{linares}. We note, however, that the conditions on the derivatives of the function $u$ were firstly noted in \cite{freire-cor} (which corrects \cite{igor-jpa}). In particular, any solution vanishing on empty sets satisfies satisfies the conditions required on $u$ in theorem \ref{thm2.3}, but the converse does not necessarily hold.

We observe that a crucial point in our approach is the existence of a norm of the solutions conserved along time. Under the conditions on the parameters, and the order $s$, in Theorem \ref{thm2.3}, we can guarantee the conservation of the $H^1(\s)$ or $L^1(\s)$ norms of the solutions and, therefore, if we can find a point $t^\ast\in(0,T)$ for which one of these norms vanishes, then it vanishes at all. Finally, we note that these ideas are essentially geometric, see \cite{igor-jpa} for a better discussion.

\section{Conclusion}\label{sec8}

In this work we found conditions for \eqref{1.0.1} to have global solutions and to avoid the breaking wave of solutions. We also proved some unique continuation properties of the periodic solutions of \eqref{1.0.1}. All of these results are proved under certain constraints on the parameters of the equation. Moreover, what is reported in the present work shed light to some open questions pointed out in \cite{himonas-jmp-2014}.

\section*{Acknowledgements}

I. L. Freire is grateful to FAPESP for financial support (grant nº 2020/02055-0).

\end{document}